\documentclass[a4paper,10pt]{scrartcl}
\usepackage[utf8]{inputenc}
\usepackage[english]{babel}
\usepackage[fixlanguage]{babelbib}
\usepackage{cite}
\usepackage{amssymb, amsmath, amstext, amsopn, amsthm, amscd, amsxtra, amsfonts}
\usepackage{yfonts, mathrsfs}
\usepackage{colonequals}
\usepackage{enumerate}
\usepackage{graphicx}%
\usepackage{colonequals}
\usepackage[all]{xy}
\usepackage{tikz}
\usepackage{hyperref}%
\hypersetup{
urlcolor = red,
colorlinks = true,
linkcolor = blue,
citecolor = blue,
linktocpage = true,
pdftitle = {Convergence of Lie group integrators},
pdfauthor = {Charles Curry, Alexander Schmeding},
bookmarksopen = true,
bookmarksopenlevel = 1,
unicode = true,
hypertexnames =false
}%

\swapnumbers
\newtheoremstyle{standard}
{16pt} 
{16pt} 
{} 
{} 
{\bfseries}
{} 
{ } 
{{\thmname{#1~}}{\thmnumber{#2.}}\thmnote{~(#3)}} 

\newtheoremstyle{kursiv}
{16pt} 
{16pt} 
{\itshape} 
{} 
{\bfseries}
{} 
{ } 
{{\thmname{#1~}}{\thmnumber{#2.}}\thmnote{~(#3)}} 

\theoremstyle{standard}
\newtheorem{defn} [subsection]{Definition}
\newtheorem{ex} [subsection]{Example}

\newtheorem{rem} [subsection]{Remark}

\newtheorem{setup} [subsection]{}

\theoremstyle{kursiv}
\newtheorem{thm}[subsection]{Theorem}

\newtheorem{lem} [subsection]{Lemma}

\newtheorem{ass} [subsection]{Assumption}

\newcommand{\tr}{\triangleright}

\newcommand{\R}{\ensuremath{\mathbb{R}}}

\newcommand{\N}{\ensuremath{\mathbb{N}}}

\DeclareMathOperator{\Fl}{Fl}

\newcommand{\coloneq}{\colonequals}

\tikzstyle dtree=[grow'=up,sibling distance=4mm,level distance=4mm,thick]
\tikzstyle dtree node=[scale=0.4,shape=circle,very thin,draw]
\tikzstyle dtree black node=[style=dtree node,fill=black]
\newcommand{\onenode}{
  \begin{tikzpicture}[dtree]
    \node[dtree black node] {}
    ;
  \end{tikzpicture}
}

\newcommand{\twonode}{
\begin{tikzpicture}[dtree]
  \node[dtree black node] {}
  child { node[dtree black node] {} }
  ;
\end{tikzpicture}
}

\newcommand{\threenode}{
\begin{tikzpicture}[dtree]
  \node[dtree black node] {}
  child { node[dtree black node] {} }
  child { node[dtree black node] {} }
  ;
\end{tikzpicture}
}

\newcommand{\threenodezwo}{
\begin{tikzpicture}[dtree]
  \node[dtree black node] {}
  child { node[dtree black node] {}  child { node[dtree black node] {} } }
  ;
\end{tikzpicture}
}
\title{Convergence of Lie group integrators}
\author{Charles Curry\footnote{NTNU Trondheim, \href{mailto:charles.curry@ntnu.no}{charles.curry@ntnu.no}}, Alexander Schmeding\footnote{TU Berlin, \href{mailto:schmeding@tu-berlin.de}{schmeding@tu-berlin.de}}}

\begin{document}

\maketitle

\begin{abstract}
We relate two notions of local error for integration schemes on Riemannian homogeneous spaces, and show how to derive global error estimates from such local bounds. In doing so, we prove for the first time that the Lie-Butcher theory of Lie group integrators leads to global error estimates.
\end{abstract}


\smallskip

 (primary);
  22F30, 
  53C30 
   (secondary)

\tableofcontents
\section*{Introduction}
The 1990s saw rapid development in numerical methods for differential equations on manifolds that are intrinsic in the sense that they preserve the manifold structure by evolving using geometric operations such as group actions and exponentials, see \cite{MR1216986,MR1431350,MR1661989,MR1621080,MR1662814}. The case where the manifold in question is a homogeneous space has received particular focus, as it allows the equation to be phrased in terms of Lie group actions, with important consequences for both practical implementation and theoretical analysis of methods.

Two classes of method in particular, the Runge-Kutta Munthe-Kaas (RKMK) methods and commutator-free methods  (a generalization of Crouch-Grossmann methods) \cite{CFREE} are now widely used in geometric integration theory and have proven themselves in a wide range of problems (see e.g.\ \cite{MR1883629} for a survey).

The order theory of such methods is founded in Lie-Butcher theory, \cite{OW06,MR2407032,MR3350091}, which generalizes the order theory of numerical methods in $\mathbb{R}^n$ rooted in B-series to homogeneous space. Whilst the algebraic side of this field has reached considerable maturity, there remain notable analytic gaps which we aim to fill. In particular, there does not seem to exist a fully satisfactory derivation of either local or global estimates at present. Indeed, to the authors knowledge, only partial results are available, such as \cite{MR1799308} cf.\ also the survey in \cite[Section 9]{MR1883629} and most recently \cite[Section 3]{1805.07578v1}. The main issue here is twofold:
\begin{itemize}
\item local error estimates, even though available, are given by different kinds of estimate (e.g.\ using test functions, using a Riemannian distance) and it is not directly apparent how these properties are related. 
\item global error estimates are largely missing or only available with additional assumptions on the vector fields or the geometry of the manifold.
\end{itemize}
In a bit more detail, a Lie-group method is said to be of order $p$, if the Taylor expansion of flow generated by it (and tested against an arbitrary smooth function) coincides up to order $p$ with the Taylor expansion of the exact solution. This immediately leads to the following local error estimate:
\smallskip

\emph{Suppose $V$ is a $C^{p+1}$ vector field on a manifold $M$, and let $\hat{y}$ be an approximation of the integral curve $y$ of $V$, of order $p$ in the above sense. Then for all $f\in C^{\infty}(M)$,}
\begin{displaymath}
|f(\hat{y})-f(y)| \leq C h^{p+1} \text{ for some constant } C.
\end{displaymath}
Our first main result (Theorem~\ref{thm: locmet:est}) clarifies how the local estimate, given with respect to test functions, implies a comparable estimate involving the Riemannian distance as is often required \cite{MR1799308,MR1883629}. Recall that the natural setting for Lie group methods is a homogeneous Riemannian manifold $(M,g)$\footnote{Every Lie group is a homogeneous Riemannian manifold. We recall the necessary results and facts on these manifolds in Section \ref{sect: Prelim}.}.For the rest of this introduction we will always assume that $(M,g)$ is such a manifold. Then our results subsume the following:
\smallskip

\textbf{Theorem A} \emph{Let $V$ be a $C^{p+1}$ vector field such that $\hat{y}$ approximates the integral curve $y$ of $V$ up to order $p$. Then the local estimate} 
$$d(y(h),\hat{y}(h)) \leq Ch^{p+1}$$
\emph{
holds, where $d$ is the geodesic distance on $M$.}\smallskip

This result settles the first problem mentioned above and clarifies the dependency on the different kinds of error estimates found in the literature. The second point is covered by Theorem~\ref{thm: global}, which shows that global error estimates follow from local estimates involving the Riemannian metric. As such it subsumes the following 
\smallskip

\textbf{Theorem B} \emph{For a $C^{p+1}$ vector field $V$ we fix a sequence $\{\hat{y}_n\}_{i=1,\ldots,n}$ approximating the integral curve of $V$ through $y_0$ at a discrete set of times $t_i$ with $h_i = t_{i+1} - t_i$, and $\max_i h_i = h$. If $\hat{y}$ obeys either of the local estimates above with exponent $p+1$, then we obtain the global estimate}
\[
d(y_n,\hat{y}_n) \leq C h^{p},
\]
We mention here that our results also clarify the dependency of the constants $C$ on the parameters such as the vector field $V$ (which we deliberately suppressed in the above statements of our theorems). These results taken together give a fully rigorous analytic counterpart to the algebraic Lie-Butcher order theory.\medskip

The paper is organized as follows: we begin with a brief overview of Riemannian homogeneous spaces, fixing notation and stating some standard results. We follow this in \textsection 2 by a brief treatment of the local estimates obtained from Lie-Butcher theory, where for the sequel it is important that we establish estimates with explicit remainder terms. The passage from local estimates obtained via Lie-Butcher theory to local estimates using the Riemannian metric is then accomplished in \textsection 3. We conclude in \textsection 4 with a derivation of the global error estimate.

\section{Preliminaries on Riemannian manifolds}\label{sect: Prelim}

In this section we fix the notation and general setting. All of the material here is standard and can be found in books on differential geometry and Riemannian geometry, e.g.\ \cite{MR1666820,MR1330918,MR1393941}. We assume that the reader is familiar with basic concepts such as Riemannian metrics and associated concepts such as (Levi-Civita) connections and covariant derivatives.

\begin{setup}
We let $\N \coloneq \{1,2,\ldots\}$ denote the natural numbers and $\N_0 \coloneq \N \cup \{0\}$. All manifolds in this paper are assumed to be paracompact and finite dimensional. By $(M,g)$ we denote a Riemannian manifold, where we write the following for the data associated to $g$:
\begin{itemize}
\item $g_m$ will be the inner product on $T_mM , m \in M$ with associated norm $\lVert \cdot \rVert_{g_m}$.
\item $\nabla X$ will be the covariant derivative of a vector field $X$.
\item $d \colon M \times M \rightarrow \R$ will be the geodesic length metric induced by $g$.
\end{itemize}
\end{setup}

In general we will be working with a special class of Riemannian manifolds, arising as quotients of isometric Lie group actions: the so called homogeneous Riemannian manifolds (see e.g.\ \cite[B.7]{MR2371700}).

\begin{setup}
By $\Lambda \colon G \times M \rightarrow M$ we denote a (left) Lie group action on $(M,g)$ such that
\begin{enumerate}
\item the action $\Lambda$ is transitive,
\item $\forall g \in G$, the map $\Lambda_g \coloneq \Lambda (g,\cdot)$ is a Riemannian isometry ($g$ is left $G$-invariant).
\end{enumerate}
Note that by the above $(M,g)$ becomes a Riemannian homogeneous space, i.e.\ the isometry group $\text{Iso} (M,g)$ acts transitively. 
This entails that as manifolds $M \cong G/H$ where $H$ is a compact subgroup of $G$ (where $H$ can be identified as the stabiliser subgroup of a point). Denote by $\pi \colon G \rightarrow G/H \cong M$ the canonical quotient map and set $o := \pi (e)$ (where $e$ is the identity element of $G$). 

Finally note that Riemannian homogeneous spaces are geodesically complete \cite[IV. Theorem 4.5]{MR1393941}, i.e.\ geodesics exist for all time. 
\end{setup}

Many manifolds appearing in applications are homogeneous Riemannian manifolds as is recalled in the following example.

\begin{ex}
\begin{enumerate}
\item Every Lie group $G$ is a Riemannian homogeneous space, where $H=\{e\}$ is the identity subgroup and $g$ is a left invariant Riemannian metric
\item Spheres \cite[B. Example 7.13]{MR2371700} and projective spaces \cite[B. Example 7.1]{MR2371700} are homogeneous Riemannian spaces (e.g.\ $\mathbb{S}^n \cong \text{SO}(n)/\text{SO}(n-1)$ where the Riemannian metric is induced by the biinvariant metric on $\text{SO}(n)$)
\end{enumerate}
We refer to the survey \cite{MR1883629} for a wealth of examples on numerical integrators on these spaces which can be treated in the framework of Lie-Butcher theory.
\end{ex}

Let us remind readers who are not familiar with Riemannian geometry that many properties of Riemannian homogeneous spaces might be conveniently formulated as properties of the geodesic length metric $d$. We collect two important facts:

\begin{setup}[Metric view of Riemannian manifolds]\label{setup: hommfd} 
\begin{enumerate}
\item The Riemannian manifold $(M,g)$ is complete if and only if $(M,d)$ is a complete metric space (this is part of the famous Hopf-Rinow theorem \cite[Theorem 1.65]{MR2371700}).
\item A surjective smooth map $f \colon M \rightarrow M$ is an isometry, if and only if it is distance preserving: $d(f(x),f(y)) = d(x,y), \forall x,y\in M$ (see \cite[Theorem 1.75]{MR2371700}).
\end{enumerate}
\end{setup}

Finally, we fix notation concerning vector fields.

\begin{setup}[Vector fields and flows]\label{setup: VF:FL}
We will denote by $\mathcal{X}^p (M)$ the space of all vector fields on $M$ of class $p \in \N \cup \{\infty\}$ (writing $\mathcal{X}(M) \coloneq \mathcal{X}^\infty (M)$).
For $V \in \mathcal{X}^{p} (M)$ we write
 \begin{displaymath}
 \Fl^V_0 \colon \R \times M \supseteq \mathcal{D} (V) \rightarrow M
 \end{displaymath}
for the \emph{flow associated to} $V$. The flow is defined by sending a pair $(t,x_0) \in \mathcal{D}(V)$ ($\mathcal{D}(V)$ is open subset of $\R \times M$) to the solution $y(t) \coloneq y_{x_0} (t) = \Fl^V_0 (x_0,t)$ of the initial value problem 
$$ y' = V(y) \quad y(0)=x_0.$$
We let $\varphi_t \coloneq \Fl^V_0 (t,\cdot)$ be the associated (local) diffeomorphism (assuming that $\mathcal{D}(V) \cap \{t\} \times M$ is non void).
 
Further, the usual conventions (cf.\ e.g.\ \cite[Section V]{MR1666820}) for the derivative operation of a vector field on a $C^k$ function $f$, i.e.\ $V(f)$ and the shorthand $V^k(f) = V(V(\cdots V(f)))$ are used throughout the text. As is common $C^p(M)$ denotes the set of real valued $C^p$-functions on the manifold $M$.

For $f\in C^p(M)$, we let $\varphi_h^*(f) = f\circ \varphi_h$ denote the pullback.    
\end{setup}

In the next section we turn now to Lie-Butcher theory. After a very brief primer on the most important concepts, we are interested in the description of Taylor expansions of exact and numerical solutions to differential equations on Riemannian manifolds.

\section{Lie-Butcher theory and Lie Series estimates}

This section is devoted to giving a precise version of the local estimates deriving from the Lie-Butcher order theory for Lie group integration methods developed in \cite{MR1621080,MR1662814,MR1682404}. 
Only an extremely brief discussion of Lie-Butcher theory is provided to lay the foundation for the following computations. For a friendly introduction to the theory of Lie group integrators we refer the reader to \cite{MR3133432}.

Let us first consider the Taylor expansion of an exact solution of a differential equation given by the vector field $V$ tested against a $C^\infty$ function.

\begin{lem}[Lie Series]
Let $V \in \mathcal{X}^{p+1}(M)$. For any $f \in C^{\infty}(M)$, the pullback action of the flow of $V$ has the Taylor series expansion
\begin{equation}\label{eq: ex:Tay}
\varphi_h^*(f)(x) = f(x) + \sum_{k=1}^{p} \frac{1}{k!} h^k V^k(f)(x) + \frac{1}{(p+1)!}h^{p+1} \varphi^*_t V^{p+1}(f)(x), 
\end{equation}
where $t<h$.
\end{lem}

\begin{proof}
Assume that $(h,x) \in \mathcal{D}(V)$ (whence $(t,x) \in \mathcal{D}(V)$ for all $0\leq t\leq h$).
It is a standard result \cite[V.\S 5 Propositions 5.2 and 5.3]{MR1666820} that 
\[
\frac{d}{dt} (\varphi_t^* f) = \varphi_t^* V(f).
\]
Now $V(f) \in C^{p+1}(M)$, and each further application of $V$ reduces the differentiability by one degree, so we can iterate this $p+1$ times to obtain $p+1$ derivatives of $\varphi^*_t f$. Fixing $x\in M$, we can view $\varphi^*_t(f)(x)$ as a function of $t$ alone, and the given result follows immediately from Taylor's theorem. 
\end{proof}

A key idea of Lie-Butcher theory is that a numerical scheme yields a Taylor series like expansion, the \emph{Lie-Butcher series}, which can be compared to the Taylor expansion \eqref{eq: ex:Tay} of the exact solution.
To give the Lie-Butcher series of an (numerical) approximation of the exact flow, we require the notion of a covariant elementary differential. A suitably general setting to define these differentials is the following:

\begin{setup}
 Suppose that $\triangleright \colon \mathcal{X}(M) \times \mathcal{X} (M) \rightarrow \mathcal{X}(M), \quad X \triangleright Y := \tilde{\nabla}_X Y$ is a binary product defined through a flat, constant torsion Koszul connection.\footnote{See \cite[Chapters II and III]{MR1393941} for basic informations on the connections used here. We write ``$\tilde{\nabla}$'' to emphasise that the connection will in general not be the Levi-Civita connection \cite[IV.2]{MR1393941} of the Riemannian manifold.} This may be extended to give a product on the enveloping algebra by 
\begin{equation}
\begin{gathered}
(XY)\tr Z = X \tr (Y \tr Z) - (X\tr Y) \tr Z \\
X \tr (YZ) = (X\tr Y) Z + Y(X\tr Z),
\end{gathered}\label{eq: postlie}
\end{equation}
see \cite{MR3350091,1804.08515v3}. Note that we can not define a similar concept for vector fields of class $C^{p}$ for $p\in \N$ as they do not form a Lie algebra (whose universal enveloping algebra the construction uses). However, given a connection as above covariant derivatives of $C^p$ vector fields make sense if one takes care to account for the loss of differentiability.   
\end{setup}

Now we note that in every term of \eqref{eq: ex:Tay}, the vector field $V$ acts (up to $p+1$ times) as a derivation on the test function $f$. Following an idea by Cayley, this situation can conveniently be described using rooted trees:

\begin{setup}{Trees}
For $n\in \N$, a \emph{rooted tree} of degree $n$ is a finite oriented
tree with $n$ vertices. We distinguish one vertex without outgoing edges, the \emph{root} of the tree.
Any vertex can have arbitrarily many incoming edges, and any vertex other than the root has exactly one
outgoing edge. Vertices with no incoming edges are called leaves. A planar rooted tree is a rooted tree together with an embedding in the plane. A planar rooted forest is a finite ordered collection of planar
rooted trees. Here the planar rooted forests are depicted up to order three (with $\emptyset$ being the empty tree):
$$\emptyset \qquad \onenode \qquad \twonode \quad \onenode \, \onenode \qquad \threenodezwo \quad \threenode \quad \twonode \, \onenode \quad \onenode \, \twonode   $$
\end{setup}

We will see now, one one can construct differential operators from forests with edges using the binary product $\triangleright$.  

\begin{defn}[Elementary covariant differentials]\label{defn: elcov}
Let $\tau = B_+(\tau_1,\tau_2,\ldots,\tau_n)$ be a planar tree defined recursively by connecting the branches $\tau_1,\ldots,\tau_n$ from left to right onto a root. We then define the \emph{elementary (covariant) differentials} of $V$ recursively by
\begin{align*}
V_{\bullet} &= V,\\
V_{\tau} &= (V_{\tau_1}\cdots V_{\tau_n})\tr V,
\end{align*}
where $\bullet$ is the tree with a single node and products are expanded via \eqref{eq: postlie}. This is extended to forests of trees $\tau_1 \cdots \tau_n$ by
\[
V_{\tau_1 \tau_2}(f) = V_{\tau_1}\big(V_{\tau_2}(f)\big) - \big(V_{\tau_1}\tr V_{\tau_2}\big)(f),
\]
see \cite[Section 3]{MR3350091} or \cite[5.2]{1804.08515v3}.
\end{defn}

\begin{rem}
\begin{enumerate}
\item The definition of the elementary differentials may differ depending on the setting, for example the connection may be defined implicitly by specification of a rigid frame, or using the action Lie algebroid structure associated to Lie group integration. The specific manner is not important for our presentation.
\item For a forest $\tau$ of order $p$, the construction of the differential operator $V_\tau$ in Definition \ref{defn: elcov} requires one to take (recursively) at most $p$ (covariant) derivatives of the vector field $V$. Thus if the order of the forest $\tau$ is smaller then $p$, we may consider the differential operator $V_\tau$ for any vector field $V \in \mathcal{X}^p(M)$.
\item Again by construction $V_\tau$ takes at most $p$ derivatives of the function $f$ if $\tau$ is a forest of order at most $p$. 
\end{enumerate}
\label{rem: diff:op}
\end{rem}

The typical assumption of Lie-Butcher order theory is that a method of order $p$ admits a Taylor expansion of the following type:

\begin{ass}[Lie-Butcher series]\label{LieButcher}
For any test function $f \in C^\infty (M)$, the pullback action of an approximate flow $\hat{y}$ of order $p$ has a Taylor series expansion
\[
\hat{\varphi}_h^*(f)(x) = f(x) + \sum_{k=1}^{p} \frac{1}{k!} h^k V^k(f)(x)
+ \sum_{|\tau|=p+1} \alpha(\tau) h^{p+1} \hat{\varphi}^*_t V_{\tau}(f), 
\]
where $|\tau|$ is the number of nodes in the forest $\tau$, $\alpha$ is a linear functional on forests and $\hat{\varphi}_h$ is the (local) diffeomorphism associated to the flow $\hat{y}$.
\end{ass}

We emphasise here that due to Remark \ref{rem: diff:op}, it makes sense to consider the Assumption \ref{LieButcher} for $V \in \mathcal{X}^{p+1}(M)$ for $p \in \N$ as the differential operators $V_\tau$ make sense in this regime. 

\begin{rem}
In a universal setting, the Lie-Butcher series of a method is often identified with a character $\alpha(\tau)$, i.e. a multiplicative linear functional on the space of planar forests with concatenation product, see \cite{MR3350091}, in particular \cite[Section 3.3]{MR3350091}. 

The Assumption \ref{LieButcher} is then satisfied whenever $V\in\mathcal{X}^{p+1}(M)$, and the $\alpha$ agrees with the exact solution character $\frac{1}{\sigma(\tau)\tau!}$ on forests with $p$ or fewer nodes (here $\sigma$ is the internal symmetry factor and $\tau!$ is the planar forest factorial character, see \cite{1804.08515v3} for definitions). 
\end{rem}

By comparing the Lie-Butcher series of an approximate flow $\hat{\varphi}$ to the Lie series of the exact flow $\varphi$, we obtain the first local estimate: 
\begin{setup}[Local estimate for smooth test functions]\label{setup: loc1}
Suppose $V\in\mathcal{X}^{p+1}(M)$, and let $\hat{y}$ be an approximation of the integral curve $y$ of $V$, of order $p$ in the sense of Assumption \ref{LieButcher}. Then for all $f\in C^{\infty}(M)$ and $y_0\in M$,
\begin{equation}\label{est:1}
|f(\hat{y})-f(y)| \leq C(V,f) h^{p+1}
\end{equation}
\end{setup}

\begin{rem}\label{rem: constant}
For later use we wish to be very explicit on how the constant $C(V,f)$ from \ref{setup: loc1} \eqref{est:1} depends on $f$.
Comparing the Lie-Butcher series and the Lie series of the exact method we see that the remainder term is given (for an order $p$ method) by 
$$R (V,f)(x) = \frac{1}{(p+1)!}\varphi^*_t V^{p+1}(f)(x) - \sum_{|\tau|=p+1} \alpha(\tau) \hat{\varphi}^*_t V_{\tau}(f).$$
where $(t,x) \in \mathcal{D}(V)$, $\varphi_t$ is the flow of the exact solution and $\hat{\varphi}_t$ the flow of the approximate solution. As $C(V,f)$ can be chosen to be any constant which dominates the norm of the remainder term, we need an estimate on this norm.
Thus
\begin{align*}
\lvert R(V,f) (x)\rvert \lesssim \sup_{t \in [0,h]} \left(\lvert V^{p+1}(f)(\varphi_t (x))\rvert + \sum_{|\tau|=p+1} |V_\tau (f)(\hat{\varphi}_t (x))|\right), 
\end{align*}
where ``$\lesssim$'' denotes an inequality up to constants which neither depend on $f$ nor on $V$. Now by definition (cf.\ Remark \ref{rem: diff:op}) $V^{p+1}$ and $V_\tau$ act on smooth functions as differential operators whose order is at most $p+1$. Thus the terms $V^{p+1}(f)(\varphi_t (x))$ and $V_\tau (f)(\hat{\varphi}_t (x))$ can be computed as linear combinations  (depending on $V, p, \tau$ and the connection $\tilde{\nabla}$) of the partial derivatives of $f$ up to order $p+1$.
Further, we note that to obtain an estimate we only need to give an upper bound on all partial derivatives of $f$ up to order $p+1$ on the compact set $\{\varphi_t (x) \mid t \in [0,h]\} \cup \{\hat{\varphi}_t (x) \mid t \in [0,h]\}$. 
\end{rem}

\section{Local estimates to local metric estimates}\label{sect: metricest}
Our first step is to show that the condition above implies the weaker (but in some senses more natural) condition.

\begin{defn}[Local metric estimate]
Let $\hat{y}$ approximate the integral curve $y$ through $y_0$ at $y(h)$ of order $p\in \N$. Then there is a constant $C= C (\hat{y},y)$ such that
\begin{equation}\label{est:2}
d(y(h),\hat{y}(h)) \leq C h^{p+1}.
\end{equation}
\end{defn}

Note that the condition \eqref{est:2} appeared in the stability analysis for Lie group methods in \cite[Section 9]{MR1883629} and the earlier work by Faltinsen \cite{MR1799308}. However, there the local metric estimate \eqref{est:2} is \textbf{assumed to hold} for a given method to enable the analysis, whereas we will deduce the validity of \eqref{est:2} in case the method satisfies local estimate \eqref{est:1}.

To prove this result we introduce a family of smooth functions which will allow us to deduce the local metric estimate \eqref{est:2} from the local estimate for smooth test functions \ref{est:1}.
To this end, we need smooth functions controlling the geodesic distance. The constructions of these functions in Lemma \ref{lem: smoothaway} and Lemma \ref{lem: smoothnear} below is somewhat technical, hence we postpone it to Appendix \ref{app: auxiliary}.
We need a smooth function which allows us to control the geodesic distance for points ``far away'' from $o$.

\begin{lem}\label{lem: smoothaway}
For $\varepsilon >0$ and $(M,g)$ a connected\footnote{Here connectedness is only needed to make sense of condition 2 in the statement, as for a non-connected manifold it is customary to set $d(x,y)=\infty$ if $x,y$ are from different connected components. Assuming that $M$ is connected is no essential restriction as we will only compare curves lying in the same connected component.} complete Riemannian manifold, there exists $F_\varepsilon \in C^\infty (M)$ with the following properties.
\begin{enumerate} 
\item $F_\varepsilon(o)=0$ and $F_\varepsilon (x) \geq 0,\ \forall x \in M$,
\item if $d(x,o)>\varepsilon$, then $F_\varepsilon (x) \geq d(x,o)$.
\end{enumerate}
\end{lem}

Note that the Riemannian distance from a fixed point is in general only continuous even if restricted to an open set away from $o$. This phenomenon (connected to the cut locus of the Riemannian manifold) prevents us from simply ``smoothing out'' the Riemannian distance at $o$ to obtain the desired smooth function. 
Furthermore, the geodesic distance $d(\cdot,o)$ is also non smooth at $o$ whence we need smooth functions controlling the distance near $o$. 

\begin{lem}\label{lem: smoothnear}
Let $\varepsilon >0$ be so small that the closure of the metric ball $B_\varepsilon^d (o)$ is contained in a manifold chart $(U,\varphi)$. Then there is $N \in \N$ and a family $\{f_n\}_{1 \leq n\leq N} \subseteq C^\infty (M)$ with the following properties 
\begin{enumerate}
\item $f_n (o)=0$,
\item if $d(x,o)< \varepsilon$ then there is $1\leq n_x\leq N$ such that $f_n (x) \geq d(x,o)$.
\end{enumerate}
\end{lem}

Note that the functions $\{f_n\}_n$ are also allowed to take negative values (which they will take on a neighborhood of $o$!). This is unavoidable if one wants to obtain smooth functions which dominate the distance (at least in some directions) and are $0$ at $o$.
We now have all technical tools assembled to prove the main result of this section: 

\begin{thm}\label{thm: locmet:est}
Let $(M,g)$ be a homogeneous Riemannian manifold with $G$-invariant Riemannian metric and $V \in \mathcal{X}^{p+1} (M)$. Assume that $\hat{y}$ approximates the integral curve $y = \Fl^V_0 (\cdot , y_0)$ up to order $p$ as in Assumption \ref{LieButcher}. Then $\hat{y}$ satisfies the local metric estimate \eqref{est:2}.
\end{thm}

\begin{proof} We proceed in several steps to obtain the estimate from the family of smooth functions constructed in Lemma \ref{lem: smoothnear} and Lemma \ref{lem: smoothaway}. To this end, let $y \colon [0,h] \rightarrow M$ be the integral curve $y (t) \coloneq \Fl^V_0 (t,y_0)$ of a $C^{p+1}$-vector field $V$ defined (at least) on the compact interval $[0,h]$ with $h>0$. By standard arguments \cite[Section IV]{MR1666820} $y$ is a $C^{p+2}$-mapping.

\paragraph{Step 1: Smooth shift from $y(t)$ to $o$:}
 Since $M$ is a homogeneous space, there exists a $C^{p+2}$-curve $\tilde{y} \colon [0,h] \rightarrow G$ which lifts $y$, i.e.\ $\pi \circ \tilde{y}= y$.\footnote{Here we use that a homogeneous space is a principal $H$-bundle, whence a $C^{p+2}$-curve admits a $C^{p+2}$ horizontal lift, cf.\ e.g. \cite[Chapter 5.1]{MR2021152}.}
Note that the lift is non-unique but the estimates will not depend on the choices made.
By construction $\Lambda_{\tilde{y}(t))^{-1}} (y(t)) = o, \quad \forall t \in [0,h]$. Using left invariance of the metric, for every $t\in [0,h]$ the map $\Lambda_{\tilde{y}(t)^{-1}}$ is an isometry, whence for $x\in M$ and $t \in [0,h]$,
\begin{equation}\label{eq: inv:met}
d(x,y(t)) = d(\Lambda_{\tilde{y}(t)^{-1}} (x)),\Lambda_{\tilde{y}(t)^{-1}} (y (t))) = d(\Lambda_{\tilde{y}(t)^{-1}} (x),o).
\end{equation}

\paragraph{Step 2: An $h$-dependent family of smooth comparison functions.}
Choose $\varepsilon >0$ such that the closure of the ball $B \coloneq B_\varepsilon^d (o)$ is contained in a chart $(U,\varphi)$.
By Lemma \ref{lem: smoothnear} we obtain a family of smooth functions $\{f_n\}_{1\leq n\leq N}$ (where we can choose $N=2^{\text{dim} M}$) which controls the geodesic distance on $B$ such that every $f_n$ vanishes in $o$.
Then Lemma \ref{lem: smoothaway} applied for the same $\varepsilon$ yields $F_\varepsilon \in C^\infty (M)$ which controls the Riemannian distance $d(\cdot,o)$ outside of $B$ and satisfies $F_\varepsilon (o) =0$.

We construct now the functions which will yield the necessary estimates: 
\begin{align*}
\omega_n \colon [0,h] \times M &\rightarrow \R,\quad (t,x) \mapsto f_n \circ \Lambda_{y(t)^{-1}} (x),\quad 1\leq n \leq N\\
\omega_{N+1} \colon [0,h] \times M &\rightarrow \R,\quad (t,x) \mapsto F_{\varepsilon} \circ \Lambda_{y(t)^{-1}} (x).
\end{align*}
By construction the $\omega_n$ are $p+2$-times continuously differentiable with respect to $t$ and every of these differentials is smooth with respect to $x$. Thus we obtain continuous maps into the space $C^\infty (M)$ endowed with the compact open $C^\infty$-topology via\footnote{Functions with the differentiability exhibited by $\omega_n$ are called $C^{p+2,\infty}$-functions in \cite{MR3342623}. Indeed that $\omega_n$ is $C^{p+2,\infty}$ follows from the chain rules in ibid. The continuity of $\omega^\vee_n$ into the locally convex space $C^\infty (M)$ is a consequence of the exponential law \cite[Theorem B]{MR3342623} which even shows that $\omega^\vee_n$ is a $C^{p+2}$ map. Since continuity is sufficient for our purposes we do not need to explain what differentiable functions into the (non normable!) space $C^\infty(M)$ are.}
$$\omega^\vee_n \colon [0,h] \rightarrow C^\infty(M), \omega^\vee (t) \coloneq \omega_n (t,\cdot) \quad 1\leq n \leq N+1.$$ 
Recall that the compact open $C^\infty$ topology is generated by the family of seminorms which control the growth of a function and up to finitely many of its derivatives on some compact subset in $M$ (cf.\ e.g.\ \cite{MR3626202} for more on topologies for $C^\infty(M)$). Since $\omega_n^\vee$ is continuous, $\omega_n^\vee ([0,h])$ is compact whence every continuous seminorm of $C^\infty (M)$ is bounded on the image of $\omega_n^\vee$. Summarising: The growth of (up to finitely many) derivatives of the functions $\omega_n^\vee (t)$ on a given compact set can be uniformly bounded in $t$.
Finally, we note that by construction $\omega_n^\vee (t)(y(t)) = 0$ for all $1\leq n\leq N+1$ and $t\in [0,h]$.

\paragraph{Step 3: The metric estimate.}
By construction, the functions obtained via Lemma \ref{lem: smoothnear} and \ref{lem: smoothaway} control the geodesic distance of a point to $o$. Hence \eqref{eq: inv:met} implies that for $d(\hat{y}(t),y(t))\geq \varepsilon$ the function $\omega_{N+1}^\vee (t)$ controls $d(\hat{y}(t),y(t))$, while for $d(\hat{y}(t),y(t))<\varepsilon$ one of the functions $\omega_n^\vee (t), 1\leq n\leq N$ controls $d(\hat{y}(t),y(t))$.
We deduce that for every $t \in [0,h]$ there is some index $n_t$ such that
\begin{align}\label{est: depend}
d(\hat{y}(t) , y (t)) \leq \omega_{n_t}^\vee (t)(\hat{y}(t)) = | \omega_{n_t}^\vee (t)(\hat{y}(t)) - \underbrace{\omega_{n_t}^\vee (t)(y(t))}_{=0}| \leq C_{\omega_{n_t}^\vee (t)} t^{p+1}
\end{align}
where the last inequality follows from the local estimate \eqref{est:1} and $C_{\omega_{n_t}^\vee}$ depends on $\omega_{n_t}^\vee (t)$ (where the other dependencies do not matter here). Since we are after a global estimate independent of $\omega_n^\vee$ and $t$ we have to recall how these constants depend on $\omega_n^\vee$. From Remark \ref{rem: constant} we know that up to some constant $A$ (depending on the Lie-Butcher series, $V$ and the initial conditions but not on the smooth function), the constants $C_{\omega^\vee_n (t)}$ can be bounded by the partial derivatives up to order $p+1$ of $\omega_n^\vee (t)$ on the compact set 
$$
K \coloneq \{y(t) \mid t \in [0,h]\} \cup \{\hat{y}(t) \mid t \in [0,h]\}\subseteq M.
$$ 
In other words, we have to control $\sup_{t\in [0,h]}\lVert \omega_n^\vee (t)\rVert_{p+1,K}$, where $\lVert \cdot\rVert_{p+1,K}$ measures the (sum of) absolute values of partial derivatives on $K$ up to order $p+1$.

Following Step 2, we know that there is a uniform bound in $t$, i.e.\ $$R \coloneq \sup_{1\leq n \leq N+1}\sup_{t \in [0,h]} \lVert \omega_n^\vee (t)\rVert_{p+1,K}  < \infty$$ and thus $\sup_{t\in [0,h]}C_{\omega^\vee_{n_t} (t)} \leq AR \equalscolon C$. 
Hence, from \eqref{est: depend} we conclude that
$$d(\hat{y}(t),y(t)) \leq Ct^{p+1}\leq Ch^{p+1} \quad \forall t \in [0,h].\qedhere$$ 
\end{proof}

Note that the main point in the proof of Theorem \ref{thm: locmet:est} was to establish a uniform bound independent of $t$. We remark that the constant $C$ obtained for the metric estimate still depends on the choices we made in the proof (e.g.\ the choice of $\varepsilon >0$). Thus the proof is a pure existence proof without any claim of optimality of $C$. Indeed, if one chooses $\varepsilon$ very small one should expect $C$ to become bigger as it is derived from an estimate of the derivatives of smooth functions which involve cut-off functions confined to the $\varepsilon$-ball.   

\section{Local to global estimates}
In this section we prove our second main result, a global error estimate for the Lie group methods.
In the last chapter we have seen that Lie group methods satisfy a (local) metric estimate with respect to the geodesic metric. We apply now a suitable version of a Gronwall type estimate for Riemannian manifolds which was first established in \cite{MR2219824}. Let us recall its statement for easy reference:

\begin{setup}
For $X \in \mathcal{X}^p (M)$ (and $p\in\N \cup \{\infty\}$) the covariant derivative induces continuous linear maps
\begin{displaymath}
\nabla X (p) \colon (T_p M, \lVert \cdot\rVert_p) \rightarrow (T_p M, \lVert \cdot\rVert_p),\quad Y_p \mapsto \nabla_{Y_p} X, \quad p\in M,
\end{displaymath}
(cf.\ \cite[Section 1.5]{MR1330918} and \cite[Section VIII, in particular VIII, \S 2 Lemma 2.3]{MR1666820}).\footnote{Covariant derivatives are often only defined for smooth vector fields. However, the $\nabla_X V$ makes sense for vector fields from $\mathcal{X}^p (M)$ (for $p\in \N$ using that $(M,g)$ is smooth) if one accounts for the loss of differentiability.}
The operator norm of these mappings will be denoted by $\lVert \nabla X(p)\rVert_g$. 
\end{setup}

\begin{setup}[{\cite[Corollary 1.6]{MR2219824}}]\label{setup: Gronwall:smooth}
Let $(M,g)$ be a connected and complete Riemannian manifold, $V \in \mathcal{X} (M)$ and $p_0,q_0\in M$. Let $S$ be a minimizing geodesic segment connecting $p_0$ and $q_0$. Choose $T > 0$ such that the flow $\Fl^X$ of the vector field $X$ is defined on $[0,T] \times S$. 
\emph{Then the integral curves $\varphi(t) \coloneq \Fl^X_t (p_0), \psi(t)\coloneq \Fl^X_t(q_0)$ with initial value $p_0$ (resp.\ $q_0$) satisfy the Gronwall type estimate}
\begin{equation}\label{est: Gronwall}
d(\varphi(t),\psi(t)) \leq d(p_0,q_0)e^{C_T t}, \quad t \in [0,T]
\end{equation}
\emph{where $C_T = \sup \{\lVert \nabla X (p)\rVert_g \mid p \in \Fl^X ([0,T]\times S)\}$.}
\end{setup} 

The Gronwall type estimate exhibited in \ref{setup: Gronwall:smooth} has been established in \cite{MR2219824} only for smooth vector fields. We wish to obtain a similar estimate for $\mathcal{X}^p (M)$ vector fields ($p\in \N \cup \{\infty\}$).

\begin{rem}[Estimate \ref{setup: Gronwall:smooth} holds for $X \in \mathcal{X}^p (M), p \in \N$.]\label{rem: CpGronwall}
To prove that the Gronwall estimate \eqref{est: Gronwall} holds also with lower differentiability of the vector field, note that \cite[Corollary 1.6]{MR2219824} is a direct consequence of \cite[Theorem 1.4]{MR2219824}. The proof of said theorem uses the differentiability class of the vector field $V$ only through an application of \cite[Proposition 1.1]{MR2219824} in the proof.
Hence if \cite[Proposition 1.1]{MR2219824} holds for vector fields in $\mathcal{X}^p(M)$ for $p\in \N$ we are done. 

\paragraph{Reexamining the proof of \cite[Proposition 1.1]{MR2219824}:} One needs differentiability of the flow map $\Fl^V_0 (t,x_0)$ (cf.\ \ref{setup: VF:FL} for details on the notation). Namely, the existence of iterated derivatives of the type $\partial_t \partial_{x_0} \Fl^V$ and $\partial_t \partial_{x_0} \Fl^V$ are required. If these second mixed partial derivatives (with respect to $t$ and $x_0$) exist, then the proof can be carried out exactly as presented in \cite[p.\ 134-135]{MR2219824}. 

\paragraph{The case $p\geq 2$.} 
For $p\geq 2$ by standard ODE arguments (cf.\ e.g.\ \cite[Section IV]{MR1666820} or \cite[Proposition 5.13]{MR3342623}) also $\Fl^V_0$ is a $C^p$ map whence the iterated differentials exist.

\paragraph{The case $p=1$.} Looking closer, the derivatives still exist for $p=1$: We call a continuous $f \colon M \times N \rightarrow L$ between (finite dimensional) manifolds a $C^{1,1}$ map if (in charts) the partial derivative with respect to $M$ exists, is continuous and $f$ and the derivative are again continuously differentiable with respect to the $N$ variable (cf.\ \cite[Definition 5.11]{MR3342623}). 
The concept of $C^{1,1}$ (or the general $C^{r,s}$ maps developed in \cite{MR3342623}) captures exactly existence of the mixed partial derivatives needed. 
Now for $V \in \mathcal{X}^1 (M)$, \cite[Proposition 5.13]{MR3342623} asserts that the flow $\Fl^V_0$ is indeed of class $C^{1,1}$. \smallskip

Summing up, the Gronwall estimate \ref{setup: Gronwall:smooth} holds $V \in \mathcal{X}^p(M)$ where $p\in \N \cup \{\infty\}$.
\end{rem}

\begin{thm}[Global error estimate]\label{thm: global}
Consider a vector field $V \in \mathcal{X}^{p+1} (M)$ with $p\in \N_0 \cup \{\infty\}$ on a Riemannian homogeneous space $(M,g)$ together with a sequence $\{\hat{y}_n\}_{i=1,\ldots,n}$ approximating the integral curve of $V$ through $y_0$ at a discrete set of times $t_i$ with $h_i = t_{i+1} - t_i$, and $\max_i h_i = h$. If $\hat{y}$ obeys either of the local estimates above with exponent $p+1$, then we obtain the global estimate
\[
d(y_n,\hat{y}_n) \leq C h^{p},
\]
where $C$ is a constant depending only on $V, y_0$ and $T$.
\end{thm}

\begin{proof}
The proof follows the standard ``Lady Windemere's fan'' argument, as per \cite[Section 2.3]{hairer2008solving}. Indeed, for $i=1,\ldots,n$, define the local error 
\[
e_{i} = d\big(\hat{y}_i,\varphi_h(\hat{y}_{i-1})\big),
\]
and the transported local error
\[
E_{i} = d\big(\varphi_{T_i}(\hat{y}_i) , \varphi_{T_{i-1}}(\hat{y}_{i-1})\big),
\]
where $T_i = T - t_i$. From \ref{setup: Gronwall:smooth} and Remark \ref{rem: CpGronwall} (adapting \cite[Corollary 1.6]{MR2219824}) the errors are related by
\[
E_i \leq e^{C_T T_i} e_i
\]
We then use the local metric estimate (\ref{est:2}), if necessary invoking Theorem~\ref{thm: locmet:est} to justify this, obtaining
\[
e_i \leq C_i h_i^{p+1}.
\]
The Lady Windemere's fan estimate then concludes the argument; indeed taking $C = \max_i C_i$ we have
\begin{align*}
d(y_n,\hat{y}_n) &\leq \sum_{i=1}^n E_i \\
&\leq h^p C \big( h_0 e^{C_T T_1} + h_1 e^{C_T T_2} + \ldots \big) \\
&\leq h^p \frac{C}{C_T} \big(e^{C_T T} - 1\big) \qedhere
\end{align*}

\end{proof}

\begin{rem}
Global error estimates for discrete gradient descent methods were recently obtained in \cite{1805.07578v1} and the methods in ibid.\ are very similar to the ones used to derive Theorem \ref{thm: global}. Though ibid.\ concerns itself with discrete gradient methods, it is not hard to see that the arguments given there are universal, i.e.\ could be adapted to analyse the convergence of general numerical methods.

The key difference between our approach and the analysis in \cite{1805.07578v1} is in the basic setting: Studying Lie group integrators we are working by default in a complete Riemannian manifold. On non complete Riemannian manifolds, the local to global argument using a Gronwall inequality \ref{setup: Gronwall:smooth} breaks down (see \cite[Example 1]{MR2219824}). Withou completeness of the manifold the argument holds in general only for complete vector fields. 
Thus the local to global result \cite[Theorem 2]{1805.07578v1} is derived only for complete and smooth vector fields (though on a possibly non complete manifold, see also \cite[Remark after Theorem 2]{1805.07578v1} on how to relax the completeness condition).
Note that in light of Remark \ref{rem: CpGronwall} also the methods in ibid.\ will work for $C^1$-vector fields instead of smooth vector fields.
\end{rem}

\appendix
\section{Auxiliary constructions}\label{app: auxiliary}

In this appendix we collect several auxiliary results which enable us to construct smooth functions needed in the estimates.
We begin with a technical Lemma concerning partitions of unity with some desirable properties:

\begin{lem}\label{lem: pou}
Let $M$ be a paracompact finite dimensional manifold, $o\in M$ and $B$ be an open $o$-neighborhood. There exists a locally finite open cover $\{U_i\}_{i\in I}$ of $M$, such that $I = J \cup \{i_o\}$ and the following holds:
\begin{enumerate}
\item $i_o$ is the unique index such that $o \in U_{i_o}$,
\item $U_{i_o} \subseteq B$,
\item every $U_{i}$ is connected and relatively compact,
\end{enumerate} 
\end{lem} 

\begin{proof}
Since $M$ is locally compact, we can choose a connected manifold chart $(U_{i_o},\varphi_{i_o})$ and compact $o$-neighborhoods $C_1,C_2$ of $o$ such that the following inclusions hold:
$$o \in C_1 \subseteq U_{i_o}  \subseteq \overline{U}_{i_o} \subseteq C_2^\circ \subseteq C_2 \subseteq B$$
(where $\overline{U}_{i_o}$ is the closure and $C_2^\circ$ the interior).
Then $U \coloneq M \setminus C_1$ is open and metrisable, whence paracompact. Following \cite[II, \S 3 Theorem 3.3]{MR1666820} there is a locally finite cover of $U$ by charts $(U_j,\varphi_j)_{j\in J'}$ such that $U_j$ is connected and relatively compact. 
Let us now throw out all elements of the cover which are contained in $U_{i_o}$, i.e.\ define $J \coloneq \{ j \in J'\mid U_j \cap M \setminus U_{i_o} \neq \emptyset\}$ and set $I \coloneq J \cup \{i_o\}$.
By construction $o \in U_i$ for $i\in I$ if and only if $i=i_o$, $U_{i_o} \subseteq B$ and every $U_i$ is connected and relatively compact. 
To prove that $\{U_i\}_{i\in I}$ is a locally finite cover of $M$, we observe that $\{U_i\}_{i\in I}$ covers $M$ by construction. Now $K \coloneq C_2 \setminus U_{i_o} \subseteq U$ is compact, whence only finitely many elements of the locally finite cover $\{U_j\}_{j\in J}$ intersect it. This means that only finitely many of the sets $U_j, j\in J$ intersect $U_{i_o}$, whence $\{U_i\}_{i\in I}$ is locally finite. 
\end{proof}

We now prove Lemma \ref{lem: smoothaway} whose statement we repeat here for  convenience.

\begin{lem}
For $\varepsilon >0$ and $(M,g)$ a connected complete Riemannian manifold, there exists $F_\varepsilon \in C^\infty (M)$ with the following properties.
\begin{enumerate} 
\item $F_\varepsilon(o)=0$ and $F_\varepsilon (x) \geq 0,\ \forall x \in M$,
\item if $d(x,o)\geq\varepsilon$, then $F_\varepsilon (x) \geq d(x,o)$.
\end{enumerate}
\end{lem}

\begin{proof}
Let $\varepsilon >0 $ and denote by $B \coloneq B_\varepsilon^d (o)$ the metric ball of radius $\varepsilon$ around $o$.
Apply now Lemma \ref{lem: pou} with the above choice of $B$ to obtain a locally finite open cover $\{U_i\}_{i\in I}$ of $M$ with a unique elemnt $U_{i_o}$ such that $o\in U_{i_o} \subseteq B$.
Following \cite[II, \S 3 Corollary 3.8]{MR1666820} we pick a smooth partition of unity $\{\chi_i\}_{i\in I}$ subordinate to the cover $\{U_i\}_{i\in I}$. Note that by construction of the cover, we must have $\chi_{i_0} (o) = 1$.
Define the constants $M_j \coloneq \max\{\varepsilon, \sup_{y \in \overline{U}_j} d(o,y)\}$ for $j\in J$. By compactness of $\overline{U}_j$ and continuity of the Riemannian distance (follows from \cite[Theorem 1.9.5]{MR1330918}), the $M_j$ are finite.
Hence we can build a family of smooth function:
$$f_{i}  (x) \coloneq \begin{cases}
  \varepsilon(1- \chi_{i_0} (x)) & \text{ for } i = i_o\\
   M_j \chi_j(x) & \text{ for } i=j\in J\end{cases} \quad i \in I.$$
Observe now that since the $\{\chi_i\}_{i\in I}$ is a partition of unity, their supports form a locally finite family $\{\text{supp} \chi_i\}_{i\in I}$. We deduce that the family of supports for the functions $f_i$ is also locally finite, whence we can define a smooth function  
$$F_\varepsilon (x) \coloneq \sum_{i \in I} f_i (x) \quad x \in M$$
which satisfies $F_\varepsilon (o)=0$ and $F_{\varepsilon}(x) \geq 0$ for all $x\in M$. If $x \in M \setminus B$, there is a finite non empty $L_x \subseteq J$ such that $x \in U_i$ if and only if $i \in L_x$.
 In particular $\sum_{i \in L_x} \chi_i (x) = 1$ and as $x \in U_i$ for every $i \in L_x$ by construction one has $d(x,o)\leq M_i$ for all $i\in L_x$.
 Thus we deduce that 
 $$d(x,o) \leq  \min_{i\in L_x} \{M_i\} \sum_{i\in L_x}  \chi_i (x) \leq   \sum_{i\in L_x} M_i \chi_i (x) \leq \sum_{i \in L_x \cup \{i_o\}} f_i(x) = F_\varepsilon (x).$$  
\end{proof}

Finally, we construct a family of smooth functions which allows us to obtain estimates on the Riemannian distance for points close to $o$. This is Lemma \ref{lem: smoothnear} whose statement we repeat for the readers convenience.

\begin{lem}
Let $\varepsilon >0$ be sufficiently small that the closure of the metric ball $B_\varepsilon^d (o)$ is contained in a manifold chart $(U,\varphi)$. Then there is $N \in \N$ and a family $\{f_n\}_{1 \leq n\leq N} \subseteq C^\infty (M)$ with the following properties 
\begin{enumerate}
\item $f_n (o)=0$,
\item if $d(x,o)< \varepsilon$ then there is $1\leq n_x\leq N$ such that $f_n (x) \geq d(x,o)$.
\end{enumerate}
\end{lem}

\begin{proof}
As a homogeneous Riemannian manifold, $(M,g)$ is complete, see \ref{setup: hommfd}. Thus the closed and bounded set $K \coloneq \overline{B_\varepsilon^d (o)}$ is compact by the Hopf-Rinow theorem  \cite[Theorem 1.65]{MR2371700}. We may assume without loss of generality that $\varphi(o)=0$.
Now by standard arguments \footnote{see eg.\ the answer by Beno\^it Kloeckner at \url{https://mathoverflow.net/a/236851/}} for every smooth Riemannian manifold the charts are locally bi-Lipschitz to Euclidean space. Since $K \subseteq U$ is compact, we may (after shrinking $U$ if necessary) assume that $\varphi$ is bi-Lipschitz with respect to the euclidean distance $d_{2}$ on $\R^n$ and the geodesic distance on $U$, i.e.\ 
\begin{equation}\label{eq: Lipnorm}
 d_{2} (\varphi (x),\varphi(y)) \lesssim d(x,y) \lesssim  d_{2}(\varphi (x),\varphi(y)),\quad \forall x,y \in U,
\end{equation}
where ``$\lesssim$'' is used to denote an inequality up to a (multiplicative) constant. Using the equivalence of norms on $\R^n$, we now replace the euclidean distance $d_{2}$ in \eqref{eq: Lipnorm} by the distance $d_1$, induced by the $\ell^1$-norm $\lVert x\rVert_1 \coloneq \sum_{i=1}^n |x_i|$. 
We claim now, that there is $N\in \N$ and a family of smooth functions $\{P_n\}_{1\leq n\leq N} \subseteq C^\infty (\varphi(U))$ which satisfy the following properties for all $1 \leq n\leq N$:
\begin{enumerate}
\item $P_n (\varphi(o)) = H_n (0) = 0$ 
\item if $x \in \varphi (K)$ then there exists $1\leq n_x \leq N$ such that  $\lVert x\rVert = d_1 (x,\varphi(o)) \leq P_{n_x} (x)$.
\end{enumerate}
If this were true, then the proof can be finished as follows: 
Let $L$ be the (smallest) Lipschitz constant such that $d(x,y) \leq Ld_1(\varphi(x),\varphi(y)), \quad \forall x,y\in U$.
Since $U$ is an open neighborhood of $K$, we can choose a smooth cut-off function $\xi \colon M \rightarrow [0,1]$ such that $\xi|_K \equiv 1$ and $\xi|_{M\setminus U} \equiv 0$. Then we set 
$$f_n \colon M \rightarrow \R ,\ x \mapsto \begin{cases} L\xi (x) \cdot P_n \circ \varphi (x) & \text{ if } x \in U \\ 0 & \text{otherwise}.\end{cases},$$
 Clearly we have $f_n \in C^\infty (M)$ and $f_n(o)= 0$ for all $1\leq n \leq N$. If $d(x,o) <\varepsilon$, then $x \in K$, whence there is $n_x \coloneq n_{\varphi(x)}$ as in property 2.\ of the family $\{P_n\}_n$ such that  
$$f_{n_x} (x) =  L\underbrace{\xi (x)}_{=1} \cdot P_{n_{\varphi(x)}} \circ \underbrace{\varphi (x)}_{\in \varphi(K)} \geq L d_1 (\varphi(x),\varphi(o)) \geq d(x,o).$$
\textbf{Proof of the claim:} We have to construct smooth functions which satisfy properties 1.\ and 2. To this end, consider for $1\leq k \leq n$ the smooth (linear) functions 
\begin{align*}p_{k,0} \colon \R^n \rightarrow \R, \quad (x_1, \ldots , x_n) \mapsto x_k, \qquad p_{k,1} \colon \R^n \rightarrow \R,\quad p_{k,1}(x) \coloneq -p_{k,0} (x)
\end{align*} 
Construct for every multiindex $\alpha = (\alpha_1, \ldots, \alpha_n) \in \{0,1\}^n$ a function 
$$P_{\alpha} (x) \coloneq \sum_{i=1}^n p_{i,\alpha_i} (x), \quad x\in \R^n$$
Set $N \coloneq 2^n$ and choose an arbitrary order of the multiindices $\alpha$ (naming the $i$th $\alpha^i$, to define the desired family $P_n \coloneq P_{\alpha^n}|_{U}$ for $1\leq n\leq N$. Obviously $P_n (0)=0$ and from the construction it is clear that the $P_n$ satisfy property 2.
\end{proof}

\addcontentsline{toc}{section}{References}

\bibliography{L_Int}

\end{document}